\documentclass[12pt,draft]{amsart}

\usepackage[margin=1in]{geometry}
\usepackage{amsmath,amssymb,amsthm,url}
\usepackage{dsfont}
\title{Bombieri-type theorem for convolution of arithmetic functions on Number field}

\author{Pranendu Darbar}
\address{Institute of Mathematical Sciences\\ CIT Campus, Taramani, Chennai 600113, India and Homi Bhabha National Institute, Training School Complex, Anushakti Nagar, Mumbai 400094, India}
\email[Pranendu Darbar]{dpranendu@imsc.res.in}

\author{Anirban Mukhopadhyay}
\address{Institute of Mathematical Sciences\\ CIT Campus, Taramani, Chennai 600113, India and Homi Bhabha National Institute, Training School Complex, Anushakti Nagar, Mumbai 400094, India}
\email[Anirban Mukhopadhyay]{anirban@imsc.res.in}

\newtheorem{theorem}{Theorem}[section]
\newtheorem{lemma}[theorem]{Lemma}
\newtheorem{corollary}[theorem]{Corollary}

\newtheorem*{theorem*}{Theorem}
\theoremstyle{remark}\newtheorem*{remark}{Remark}

\numberwithin{equation}{section}

\renewcommand{\pmod}[1]{\left(\mathrm{mod}\,#1\right)}
\renewcommand{\phi}{\varphi}

\newcommand{\fa}{\mathfrak a}

\newcommand{\fu}{\mathfrak u}

\newcommand{\fp}{\mathfrak p}
\newcommand{\fq}{\mathfrak q}

\newcommand{\OK}{\mathcal{O}_K}
\begin{document}

\begin{abstract}
 Let $K$ be an imaginary quadratic number field of class number one and $\mathcal{O}_K$ be its ring of integers. We show that, if the arithmetic functions $f, g:\mathcal{O}_K\rightarrow \mathbb{C}$ both have level of distribution $\vartheta$ for some $0<\vartheta\leq 1/2$ then the Dirichlet convolution $f*g$ also have level of distribution $\vartheta$.   
\end{abstract} 

\maketitle

\section{Introduction and statements of results}
Let $\Lambda(n)$ be the usual Van-Mangoldt function. For $x>1$  Siegel-Walfisz theorem states that for any $D>0$ 
\[
\sum_{n\leq x}\chi(n)\Lambda(n)=O\left(\frac{x}{(\log x)^D}\right)
\] 
for any non-principal character $\chi\pmod q$ if $q\ll (\log x)^{3D}$.

An arithmetic function $f$ is said to have {\it level of distribution} $\vartheta$ for $0<\vartheta\le 1$ if for any $A>0$
there exists a constant $B=B(A)$ such that
\begin{equation} \label{eq:1.17}
\sum_{q\leq \frac{N^{\vartheta}}{(\log N)^{C}}}\max_{M\leq N} \max_{\substack{a\\ (a,q)= 1}} 
\left|  \sum_{\substack{n\leq M \\ n \equiv a \pmod q}} f(n) -\frac{1}{\varphi(q)}\sum_{\substack{n\leq M \\ (n, q)=1}}f(n)\right|
   \ll_{A} \, \frac{N}{(\log N)^A}.
\end{equation}
 The Bombieri-Vinogradov theorem states indicator function of primes have level of distribution $\vartheta$ for any $\vartheta \leq 1/2$ and the Elliott-Halberstam conjecture predicts the level of distribution to be $1.$
 
A complex valued arithmetic function $f$ is said to satisfy Siegel-Walfisz condition if there exist positive constants $C, D$ such that
\begin{align}\label{siegel walfisz}
f(n)=O\left(\tau(n)^C\right) \quad 
\text{and} \quad \sum_{n\leq x}f(n)\chi(n)=O\left(\frac{x}{(\log x)^D}\right),
\end{align}
for any non-principal Dirichlet character $\chi\pmod q$ where $\fq$ is an ideal of $\OK$ of norm $q\ll (\log x)^{3D}.$ 

If arithmetic function $f$ and $g$ both satisfies \eqref{siegel walfisz} condition and have level of distribution $1/2$ then Motohashi \cite{MOT} obtained that the Dirichlet convolution $f*g$ does so.

In this article, we extend Motohashi's \cite{MOT} result to arithmetic functions on imaginary quadratic number fields of class number one.

Let $K$ be a number field of degree $d$, class number one with $r_1$ real and $r_2$ non-conjugate complex embeddings and $\OK$ be its ring of integers.
An element $w\in \mathcal{O}_K$ is said to be a prime number in $K$, if the principal ideal $w\mathcal{O}_K$ is a prime ideal.
Let $\mathcal P$ be the set of prime numbers in $K$.

Now we first introduce the notion of Siegel-Walfisz condition and level of distribution in number field.

For $Y'\geq 1, Y\geq 0$ and $N>1,$ let $A_b^{0}(Y', Y, N)$ be the set of $\xi\in \mathcal{O}_K$ which satisfies $Y'\leq \sigma(\xi)\leq Y+N^b$ for all real embeddings and $Y'\leq |\sigma(\xi)|\leq Y+N^b$ for all complex embeddings of $K.$  We also define $A^0(Y', N+Y')=A_{1}^{0}(Y', Y, N)$ and $A^0(N)=A^0_1(1, 0, N).$ 

A complex valued arithmetic function $f:\mathcal{O}_K\rightarrow \mathbb{C}$ is said to satisfy Siegel-Walfisz condition if there exist positive constants $C, D$ such that
\begin{align*}
\tag{S-W}\label{SW condition number field} f(\fa)=O\left(\tau(\fa)^C\right) \quad 
\text{and} \quad \sum_{\fa\in A^{0}(N)}f(\fa)\chi(\fa)=O\left(\frac{|A^{0}(N)|}{(\log N)^{3D}}\right),
\end{align*}
for any non-principal Dirichlet character $\chi\pmod \fq$ with $|\fq|\ll (\log N)^D.$

An arithmetic function $f:\mathcal{O}_K\rightarrow \mathbb{C}$ is said to have {\it level of distribution} $\vartheta$ for $0<\vartheta\le 1$ if for any $A>0$
there exists a constant $B=B(A)$ such that if $Q=\frac{|A^0(N)|^{\vartheta}}{(\log N)^B}$ then
  \begin{align}\label{LOD on number field}
  \sum_{|\mathfrak{q}|\le Q}\max_{M\leq N}\max_{(\gamma, \fq)=1}|\varepsilon(M;\fq,\fa;f)| \ll_{A,K} |A^{0}(N)| (\log N)^{-A},
  \end{align}
where \[\varepsilon(M;\fq,\fa;f) = 
 \sum_{\substack{ \fa\in A^0(M)  \\ \fa \equiv \gamma \pmod {\fq}}} f(\fa) -
   \frac{1}{\phi(\fq)} \sum_{\substack{ \fa\in A^0(M) \\ (\fa,\fq)=1} } f(\fa).\] 

An analogue of Elliott-Halberstam conjecture for number fields predicts that the prime element in $\OK$ have level of distribution $\vartheta$ with 
 any $\vartheta$ in $0<\vartheta \le 1$.
Hinz \cite{HIN1} showed that primes have level of distribution $1/2$  in totally real algebraic number fields
and have level of distribution $2/5$ in imaginary quadratic fields. 
Huxley \cite{ HUX} obtained level of distribution $1/2$ for an weighted version of 
\eqref{LOD on number field}.

\begin{remark}
Method applied in this paper relies on the equality $|\sigma(w)|=|w|^{1/2}$ for each $w\in \OK$ and embeddings $\sigma:K\rightarrow \mathbb{C}$ where $|w|$ denoted the norm of $w$. In general for a number field of degree $d>1,$ a lemma of Siegel \cite{SIE} gives the existance of two positive constants  $c_{1}$ and $c_{2}$ depending only on $K$
 with $c_1c_2=1$ and a unit $\epsilon$ of $K$ such that the inequalities 
\begin{align*}
c_{1}|\alpha|^{1/d}\leq |\sigma{(\alpha)}\sigma{(\epsilon)}|\leq c_{2}|\alpha|^{1/d} 
\end{align*}
holds for all $\alpha\in \OK$ and all embeddings $\sigma$ of $K$.
Now $c_1=c_2=1$ implies that all embeddings give equivalent norms. 
This is possible only in imaginary quadratic number fields.
\end{remark}

The following theorem is a number field version of a general result by Motohashi \cite{MOT}.

The main theorems of this paper are as follows.
\begin{theorem}\label{thm:Motohashi number field version}
Let $K$ be an imaginary quadratic field of class number one and $\zeta_{0}$ be a generator of the group of  roots of unity. 
Let $f$ and $g$ be complex valued arithmetic functions on $\mathcal{O}_K$ satisfying $f(\zeta_{0}^r\fa)=f(\fa), g(\zeta_{0}^r\fa)=g(\fa)$  for all $\fa\in \OK$ and positive integer $r$.
If $f$ and $g$ both satisfies \eqref{SW condition number field}  and have common level of distribution $1/2$  then their  Dirichlet convolution $f*g$ also satisfies \eqref{SW condition number field}  and have common level of distribution $1/2$ .
\end{theorem}
The following corollary is an iterative version of the above Theorem \ref{thm:Motohashi number field version}.
\begin{corollary}\label{corollary1}
Let $K$ be an imaginary quadratic field of class number one.
Let $f_i$ $(i=1,\ldots, n)$ be complex valued arithmetic functions on ring of integers $\mathcal{O}_K$ having common level of distribution $1/2$ such that $f_i(\zeta_{0}^{r}\fa)=f_i(\fa)$ for all $r$ and satisfies \eqref{SW condition number field} . Then the Dirichlet convolution $f_1*\ldots*f_n$ also satisfies \eqref{SW condition number field}  and have common level of distribution $1/2$.
\end{corollary}

Another application of  Theorem \ref{thm:Motohashi number field version} is with $f=\mathds{1}_{w}$ the indicator function which takes value $1$ if $w$ is a prime element in $A^0(N)$ otherwise it is $0$.

\begin{corollary}\label{corollary2}
Let $K$ be an imaginary quadratic field of class number one.
If primes in $A^0(N)$ have level of distribution $1/2$ then product of two primes in $A^0(N)$ also have  level of distribution  $1/2$.
\end{corollary}

\begin{remark}
If $f$ and $g$ have level of distribution $\vartheta$ for $0<\vartheta\leq 1/2$ then it is clear from the proof of the Theorem \ref{thm:Motohashi number field version} that $f*g$ also have level of distribution $\vartheta$.
\end{remark}

Hinz \cite{HIN1} showed  that primes have level of distribution $2/5$ in imaginary quadratic number field. 
Using this result, an application of Corollary \ref{corollary2} we get the following.

\begin{corollary}\label{corollary3}
Let $K$ be an imaginary quadratic field of class number one.
Then product of two primes in $A^0(N)$  have level of distribution $2/5$.
\end{corollary}

\section{Preliminary Lemmas}

The following lemma is Perron summation formula in imaginary quadratic number field.
\begin{lemma}\label{lem:perron1}
Let $K$ be an imaginary quadratic number field.
Let $\tilde{f}$ be complex valued arithmetic functions on ring of integers $\mathcal{O}_K$ such that $\tilde{f}(\zeta_{0}^{*}\fa)=\tilde{f}(\fa)$ and satisfy $\tilde{f}(\fa)=O\left(\tau(\fa)^C\right)$. Then for $k\geq 1$ we have,
\[
\sum_{w\in A^0(N)}\tilde{f}(w)\log ^k\left(\frac{N^2}{|w|}\right)=\frac{w_K k!}{2\pi i}\int_{\sigma-iT}^{\sigma+iT}\tilde{F}(s)\frac{(N^2)^s}{s^{k+1}}ds+O\left(\frac{N^2}{T^k}\right)
\]
where $w_K$ is the number of roots of unity of $K$, $\sigma=\Re e(s)>1$ and 
\[
\tilde{F}(s)=\sum_{w\in \mathcal{O}_K}\frac{\tilde{f}(w)}{|w|^s}.
\]
\end{lemma}
\begin{proof}
We know that the number of roots of unity in imaginary quadratic field is $2, 4$ or $6$.
Using this and a Theorem from Tenenbaum [page 134, \cite{TEN}] we have
\begin{align*}
\sum_{w\in A^0(N)}\tilde{f}(w)\log ^k\left(\frac{N^2}{|w|}\right)&=w_K\sum_{|w|\leq N^2}\tilde{f}(w)\log ^k\left(\frac{N^2}{|w|}\right)\\
&=\frac{w_K k!}{2\pi i}\int_{\sigma-iT}^{\sigma+iT}\tilde{F}(s)\frac{(N^2)^s}{s^{k+1}}ds+O\left(\frac{N^2}{T^k}\right).
\end{align*}
\end{proof}

We next state the large sieve inequality for number field $K$ of degree $d>1$. Let, $\theta_1, \cdots, \theta_d$ be an integral basis of $K$ so that every integer $\xi$ of $K$ is representable uniquely as 
\[
\xi=n_1\theta_1+\cdots +n_d\theta_d
\] 
where $n_1, \cdots, n_d$ are rational integers.

If we take an element say $\xi\in A^0(Y', N+Y')$ then as we take a fixed integral basis $\theta_1, \cdots, \theta_d$ of $K$ so the element $\xi$ can be written as 
\[
\xi=n_1\theta_1+\cdots +n_d\theta_d
\] 
where $C_2 Y'<|n_i|<C_1(Y'+N), i=1,2, \cdots, d$ and $C_1, C_2$ are depending on $K.$

\begin{lemma}[\cite{HIN2}]\label{lem:large sieve inequality in number field}
Let, $f(x)$ be a positive decreasing continuous function on $Q_1<x\leq Q_2.$ Then we have,
\begin{align*}
&\sum_{Q_1<|\fq|\leq Q_2}f(|\fq|)\frac{|\fq|}{\phi(\fq)}\sum_{\chi(\fq)}^{*}|\sum_{\xi\in A^0(Y', N+Y')}c(\xi)\chi(\xi)|^2 \\
& \ll \left(f(Q_1)(Q_1^2+|A^0(N)|)+\int_{Q_1}^{Q_2}xf(x)dx\right)\sum_{\xi\in A^0(Y', N+Y')}|c(\xi)|^2
\end{align*}
where $\sum^{*}$ denotes summation over primitive multiplicative characters $\chi\pmod \fq .$
\end{lemma}
As an application of the Lemma \ref{lem:large sieve inequality in number field} with $f(x)=1/x$ we get the following lemma.
\begin{lemma}\label{lem:main large sieve inequality}
For any positive numbers $Q_1$ and $Q_2$ with $Q_1<Q_2$ we have,
\[
\sum_{Q_1<|\fq|\leq Q_2}\frac{1}{\phi(\fq)}\sum_{\chi(\fq)}^{*}|\sum_{\xi\in A^0(Y', N+Y')}c(\xi)\chi(\xi)|^2 \\
 \ll \left(\frac{|A^0(N)|}{Q_1}+Q_2\right)\sum_{\xi\in A^0(Y', N+Y')}|c(\xi)|^2
\]
where $\sum^{*}$ denotes summation over primitive multiplicative characters $\chi\pmod \fq .$
\end{lemma}

The following lemma is a consequence of Minkowski's lattice point theorem (see \cite[page 12]{CAS}).
\begin{lemma}\label{lem:size of A(N)}
Let $A^{0}(N)$ be defined as above. We have,
 \[ 
  |A^0(N)|=(1+o(1)) \frac{(2\pi)^{r_2}N^d}{\sqrt{|D_K|}}
 \]
where $D_K$ is the discriminant of number field $K$ of degree $d.$
\end{lemma}

\begin{lemma}\label{lem:merten's theorem on number field}
Let $K$ be an algebraic number field. For any natural number $R,$ we have
\[
\sum_{\substack{\fu\subseteq \mathcal{O}_K \\ |\fu|<R}}\frac{1}{\mid\fu\mid} \ll_{K} \log R,
\]
and
\[
\sum_{\substack{\fp\in \mathcal{P} \\ |\fp|\leq R}}\frac{1}{|\fp|}\ll_{K} \log \log R 
\]
where first sum is over all non-zero integral ideals of $\OK$ whose norm is less than or equal to R.
\end{lemma}

\section{Proof of Theorem \ref{thm:Motohashi number field version}}
\begin{proof}
We assume that $M>N^{1/2}.$
For $M\leq N$ and $(\gamma, \fa)=1$ we have
\begin{align*}
\varepsilon(M;\fq,\fa;f*g)=\sum_{\substack{\xi\eta \in A^0(M)\\ \xi\eta\equiv \fa (\fq)}}f(\xi)g(\eta)-\frac{1}{\phi(\fq)}\sum_{\substack{\xi\eta \in A^0(M)\\ (\xi\eta, \fq)=1}}f(\xi)g(\eta).
\end{align*}

Now, since $\xi\eta\in A^0(M)$, we can divide the range of summation over $\xi$ and $\eta$ as follows:
$|\xi|\leq (\log N)^{A'}, (\log N)^{A'}< |\xi|\leq |A^0(M)|(\log N)^{-B'}, |\eta|\leq (\log N)^{B'}.$ 

Therefore using $|A^0\left(\frac{M}{|\eta|^{1/2}|}\right)|=(1+o(1))\frac{1}{|\eta|}|A^0(M)|$ the term $\varepsilon(M;\fq,\fa;f*g)$ can be written as 
\begin{align*}
&\varepsilon(M;\fq,\fa;f*g)= \sum_{\substack{|\xi|<(\log N)^{A'}\\(\xi, \fq)=1}}f(\xi)\varepsilon\left(\frac{M}{|\xi|^{1/2}};\fq,\xi^{-1}\fa;g\right)\\
&+ \sum_{\substack{(\log N)^{A'}< |\xi|\leq |A^0(M)|(\log N)^{-B'}\\(\xi, \fq)=1}}f(\xi)\varepsilon\left(\frac{M}{|\xi|^{1/2}};\fq,\xi^{-1}\fa;g\right)\\
&+  \sum_{\substack{|\eta|\leq (\log N)^{B'}\\(\eta, \fq)=1}}g(\eta)\left\lbrace\varepsilon\left(\frac{M}{|\eta|^{1/2}};\fq,\eta^{-1}\fa;f\right)-\varepsilon\left(\min\left(\frac{M}{|\eta|^{1/2}}, \frac{N}{(\log N)^{A'/2}}\right);\fq,\eta^{-1}\fa;f\right)\right\rbrace\\
&=:\Sigma_1+\Sigma_2+\Sigma_3.
\end{align*}
Since $|\xi|\leq (\log N)^{A'},$ using \eqref{LOD on number field} and $|A^0\left(\frac{N}{|\xi|^{1/2}|}\right)|=(1+o(1))\frac{1}{|\xi|}|A^0(N)|,$ by taking summation over norm of $\fq$ of the sum $\Sigma_1$ we have,
\begin{align*}
\sum_{|\mathfrak{q}|\le Q}\max_{M\leq N}\max_{(\gamma, \fq)=1}\Sigma_1 & \ll \sum_{\substack{|\xi|<(\log N)^{A'}}}|f(\xi)|\sum_{\substack{|\mathfrak{q}|\le \frac{|A^0\left(N/|\xi|^{1/2}\right)|^{1/2}}{\left(\log \left(N/|\xi|^{1/2}\right)\right)^{B'}}\\ (\xi, \fq)=1}}\max_{M\leq N}\max_{(\gamma, \fq)=1}\varepsilon\left(\frac{M}{|\xi|^{1/2}};\fq,\xi^{-1}\fa;g\right)\\
& \ll \sum_{\substack{|\xi|<(\log N)^{A'}}}\tau(\xi)^C \frac{|A^0(N)|}{|\xi|\log^A\left(N/|\xi|^{1/2}\right)}\ll \frac{|A^0(N)|}{(\log N)^{D'}}
\end{align*}
where $D'$ is a constant depending  on $A'$ and $C.$\\
Similarly as above we have,
\[
\sum_{|\mathfrak{q}|\le Q}\max_{M\leq N}\max_{(\gamma, \fq)=1}\Sigma_3\ll \frac{|A^0(N)|}{(\log N)^{D'}}.
\]
Therefore finally we have to estimate the following sum:
\[
\Sigma_4=\sum_{|\mathfrak{q}|\le Q}\max_{M\leq N}\max_{(\xi, \fq)=1} \sum_{\substack{(\log N)^{A'}< |\xi|\leq |A^0(M)|(\log N)^{-B'}\\(\xi, \fq)=1}}f(\xi)\varepsilon\left(\frac{M}{|\xi|^{1/2}};\fq,\xi^{-1}\fa;g\right).
\]
Now using the orthogonality of characters in algebraic number field the innermost sum of $\Sigma_4$ can be written as 
\begin{align*}
\sum_{\substack{A_1< |\xi|\leq A_2\\(\xi, \fq)=1}}f(\xi)\varepsilon\left(\frac{M}{|\xi|^{1/2}};\fq,\xi^{-1}\fa;g\right)=\frac{1}{\phi(\fq)}\sum_{\chi\neq \chi_{o}}\bar{\chi}(\fa)\sum_{\substack{A_1< |\xi|\leq A_2\\(\xi, \fq)=1}}f(\xi)\chi(\xi)\sum_{\substack{w\in A^0\left(\frac{M}{|\xi|^{1/2}}\right)\\ (w, \fq)=1}}\chi(w)g(w)
\end{align*}
where $\chi_{o}$ be the principal character $\pmod\fq$ and $A_1:=(\log N)^{A'}, A_2:=\frac{|A^0(M)|}{(\log N)^{B'}}.$
Therefore, using this estimation, the sum $\Sigma_4$ can be written as 
\begin{align*}
&\Sigma_4=\sum_{|\mathfrak{q}|\le D_1}\max_{M\leq N}\max_{\xi(\fq)}\frac{1}{\phi(\fq)}\sum_{\chi\neq \chi_{o}(\fq)}\bar{\chi}(\fa)\sum_{\substack{A_1< |\xi|\leq A_2\\(\xi, \fq)=1}}f(\xi)\chi(\xi)\sum_{\substack{w\in A^0\left(\frac{M}{|\xi|^{1/2}}\right)\\ (w, \fq)=1}}\chi(w)g(w)\\
&+ \sum_{D_1<|\mathfrak{q}|\leq Q}\max_{M\leq N}\max_{\xi(\fq)}\frac{1}{\phi(\fq)}\sum_{\chi\neq \chi_{o}}\bar{\chi}(\fa)\sum_{\substack{A_1< |\xi|\leq A_2\\(\xi, \fq)=1}}f(\xi)\chi(\xi)\sum_{\substack{w\in A^0\left(\frac{M}{|\xi|^{1/2}}\right)\\ (w, \fq)=1}}\chi(w)g(w)=:\Sigma_5+\Sigma_6.
\end{align*}
where $D_1:=(\log N)^{B}.$

To calculate the sum $\Sigma_5$ we will use \eqref{SW condition number field} condition and \eqref{LOD on number field} directly for each arithmetic functions $f$ and $g$ and for calculating sum $\Sigma_6$ we will use Lemma \ref{lem:perron1} together with large sieve inequality for algebraic number field by extracting primitive characters from the sum over all non-principal characters $\pmod \fq$.
 
\subsection*{Estimation of $\Sigma_{5}$}
Using \eqref{SW condition number field} condition and \eqref{LOD on number field}
we have
\begin{align*}
\Sigma_5 &\leq \sum_{|\fq|\leq D_1}\frac{1}{\phi(q)}\sum_{\chi\neq \chi_{o}}|\sum_{A_1<|\xi|\leq A_2}f(\xi)\chi(\xi) \sum_{w\in A^0\left(\frac{N}{|\xi|^{1/2}}\right)}g(w)\chi(\xi)|\\
&+ \sum_{|\fq|\leq D_1}\frac{1}{\phi(q)}\sum_{\chi\neq \chi_{o}}|\sum_{\substack{A_1<|\xi|\leq A_2\\(\xi)|\fq}}f(\xi)\chi(\xi) \sum_{\substack{w\in A^0\left(\frac{N}{|\xi|^{1/2}}\right)\\(w)|\fq}}g(w)\chi(\xi)|\\
&\ll D_1 \sum_{A_1<|\xi|\leq A_2}\tau(\xi)^{C}\frac{|A^0(N)|}{|\xi|\log^{B'}\left(N/|\xi|^{1/2}\right)}+(\log N)^{d'}\ll \frac{|A^0(N)|}{\log^{B'} N}
\end{align*}
for some sufficiently large constant $B'$ depending on $B$ and $C.$

\subsection*{Estimation of $\Sigma_{6}$}
To calculate sum $\Sigma_6$ we have to calculate the following sum.
\begin{align*}
&\Sigma_6^{'}:=\\
&\sum_{D_1<|\mathfrak{q}|\leq Q}\max_{M\leq N}\max_{\xi(\fq)}\frac{1}{\phi(\fq)}\sum_{\chi\neq \chi_{o}}\bar{\chi}(\fa)\sum_{\substack{A_1< |\xi|\leq A_2\\(\xi, \fq)=1}}f(\xi)\chi(\xi)\sum_{\substack{w\in A^0\left(\frac{M}{|\xi|^{1/2}}\right)\\ (w, \fq)=1}}g\chi(w)\log^2 \left(\frac{|A^0(M)|}{|\xi||w|}\right).
\end{align*}
First we will show that $\Sigma_6^{'}=O\left(\frac{|A^0(N)|}{(\log N)^{D'}}\right)$ for some large $D^{'}>2$ and then by using partial summation formula we have, $\Sigma_6=O\left(\frac{|A^0(N)|}{(\log N)^{D'-2}}\right).$

Each character $\chi\neq \chi_{o}$ occurring here is induced by a primitive character $\chi^{*}(\fq_1)$ with $\fq_1|\fq.$
So $\Sigma_6^{'}$ can be written as $\Sigma_6^{'}=$
\begin{align*}
\sum_{D_1<|\mathfrak{q}|\leq Q}\max_{M\leq N}\max_{\fa(\fq)}\frac{1}{\phi(\fq)}\sum_{\fq_1|\fq}\sum_{\substack{\chi(q_1)\\ \left(a,\frac{\fq}{\fq_1}\right)=1}}^{*}\bar{\chi}(\fa)\sum_{\substack{A_1< |\xi|\leq A_2\\(\xi, \frac{\fq}{\fq_1})=1}}f(\xi)\chi(\xi)\sum_{\substack{w\in A^0\left(\frac{M}{|\xi|^{1/2}}\right)\\ (w, \frac{\fq}{\fq_1})=1}}g\chi(w)\log^2 \left(\frac{|A^0(M)|}{|\xi||w|}\right)
\end{align*}
Writing $\fq_1\fq_2=\fq$ and using Lemma \ref{lem:merten's theorem on number field} we have
\begin{align*}
\Sigma_6^{'}\ll \log N \max_{M\leq N}\max_{|\fq_2|\leq Q}I_{M,\fq_2}
\end{align*}
where
\begin{align*}
I_{M,\fq_2}:=\sum_{D_1<|\mathfrak{\fq_1}|\leq Q}\frac{1}{\phi(\fq_1)}\sum_{\substack{\chi(\fq_1)}}^{*}\Big|\sum_{\substack{A_1< |\xi|\leq A_2}}\tilde{f}(\xi)\chi(\xi)\sum_{\substack{w\in A^0\left(\frac{M}{|\xi|^{1/2}}\right)}}\tilde{g}(\xi)\chi(w)\log^2 \left(\frac{|A^0(M)|}{|\xi||w|}\right)\Big|
\end{align*}
and $\tilde{f}(\xi)=f(\xi), \tilde{g}(\xi)=g(\xi),$ if $(\xi, \fq_2)=1,$ $\tilde{f}(\xi)=\tilde{g}(\xi)=0$ otherwise.

By using  $|A^0\left(\frac{M}{|\xi|^{1/2}|}\right)|=(1+o(1))\frac{1}{|\xi|}|A^0(M)|$, Lemma \ref{lem:size of A(N)} and Lemma \ref{lem:perron1} to the innermost sum of $I_{M,\fq_2}$ we have,
\begin{align*}
&I_{M,\fq_2}=\frac{w_K}{\pi}\sum_{D_1<|\mathfrak{\fq_1}|\leq Q}\frac{1}{\phi(\fq_1)}\sum_{\substack{\chi(\fq_1)}}^{*}\Big|\int_{\sigma-iT}^{\sigma+iT}\tilde{f}(\chi, s)\tilde{g}(\chi, s)\frac{|A^0(M)|^s}{s^3}ds\Big|+O\left(\frac{|A^0(M)|^2}{T^2}\right)\\
&=\frac{w_K}{\pi}\sum_{D_1<|\mathfrak{\fq_1}|\leq Q}\frac{1}{\phi(\fq_1)}\sum_{\substack{\chi(\fq_1)}}^{*}|I_1+I_2+I_3|+O\left(\frac{|A^0(M)|^2}{T^2}\right)=:I_{M,\fq_2}^1+I_{M,\fq_2}^2+I_{M,\fq_2}^3+E,
\end{align*}
where 
\begin{align*}
&I_1:=\int_{\sigma-iT}^{\sigma+iT}\tilde{f}(\chi, s)\tilde{g}(\chi, s)\frac{|A^0(M)|^s}{s^3}ds,\quad
\tilde{f}(\chi, s)=\sum_{\substack{A_1< |\xi|\leq A_2}}\frac{\tilde{f}(\xi)\chi(\xi)}{|\xi|^s} \quad \text{and} \quad \\
& \tilde{g}(\chi, s)=\sum_{w\in \mathcal{O}_K}\frac{\tilde{g}(w)\chi(w)}{|w|^s}, \quad \sigma=1+\frac{1}{2\log N}.
\end{align*}
For the above choice of $\sigma$ it is easy to see that for some $Y<T$,
\[
\tilde{g}(\chi, s)\ll (1+|s|)Y^{1-\sigma}(\log Y)^{d'}\quad \text{and} \quad \tilde{g}(\chi, s)\ll A_2^{1-\sigma}(\log A_2)^{d'}+A_1^{1-\sigma}(\log A_1)^{d'}.
\]
Therefore integrals $I_2$ and $I_3$ are bounded above by
\[
\ll Y^{1-\sigma}(\log Y)^{d'}(\log N)^{d'}\frac{|A^0(N)|}{T}.
\]
Write, $I_1=\int_{\sigma-iT}^{\sigma+iT}\tilde{f}(\chi, s)(\tilde{g_1}(\chi, s)+\tilde{g_2}(\chi, s)+\tilde{g_3}(\chi, s))\frac{|A^0(M)|^s}{s^3}ds =: I_4+I_5+I_6$,\\
where
\begin{align*}
&\tilde{g}_1(\chi, s):=\sum_{w; |w|\leq Y}\frac{\tilde{g}(w)\chi(w)}{|w|^s}, \quad \tilde{g}_2(\chi, s):=\sum_{w; Y<|w|\leq T}\frac{\tilde{g}(w)\chi(w)}{|w|^s} \quad \text{and}\\
&\tilde{g}_3(\chi, s):=\sum_{w; |w|>T}\frac{\tilde{g}(w)\chi(w)}{|w|^s}.
\end{align*}
By using calculations of integrals $I_2$ and $I_3$ we can say that the integral $I_6$ also bonded above by
\[
\ll T^{1-\sigma}(\log T)^{d'}(\log N)^{d'}|A^0(N)|.
\]
By splitting intervals $[D_1, Q]$ and $[A_1, A_2]$ into Dyadic intervals we have,
\begin{align} \label{Dyadic inequality}
I_{M, \fq_2}^k\ll \sum_{j=0}^{J}\sum_{i=0}^{I}I_{M, \fq_2}^k(j, i),
\end{align}
where $2^{J}D_1<Q\leq 2^{J+1}D_1, 2^{I}A_1<A_2\leq 2^{I+1}A_1,$
\begin{align*}
&I_{M, \fq_2}^k(j, i)=\int_{\sigma-iT}^{\sigma+iT}S^k_{j, i}(s)|A^0(M)|^\sigma\frac{|ds|}{|s|^3}\quad \text{and}, \quad \tilde{f}_i(\chi, s)=\sum_{2^iA_1<|\xi|\leq 2^{i+1}A_1}\frac{\tilde{f}(\xi)\chi(\xi)}{|\xi|^s}\\
&S^k_{j, i}(s):=\sum_{2^jD_1<|\fq|\leq 2^{j+1}D_1}\frac{1}{\phi(\fq)}\sum_{\chi\neq \chi_{o}}\big|\tilde{f}_i(\chi, s)\tilde{g}_k(\chi, s)| \quad (k=1,2).
\end{align*}
Observe that,
\begin{align*}
&\int_{\sigma-iT}^{\sigma+iT}\tilde{f}_i(\chi, s)\tilde{g}_1(\chi, s)\frac{|A^0(M)|^s}{s^3}ds-\int_{1/2-iT}^{1/2+iT}\tilde{f}_i(\chi, s)\tilde{g}_1(\chi, s)\frac{|A^0(M)|^s}{s^3}ds\\
&=O\left(\frac{|A^0(M)|}{T^3}\left(\sum_{w; |w|\leq Y}\frac{\tau(w)^C}{|w|^{1/2}}\right)\left(\sum_{\xi; |\xi|\leq A_2}\frac{\tau(\xi)^C}{|\xi|^{1/2}}\right)\right)\\
&=O\left(\frac{|A^0(M)|^{3/2}\sqrt{Y}(\log Y)^{d'}(\log N)^{d'}}{T^3}\right).
\end{align*}
Therefore using above observations we have
\begin{align}\label{Dyadic inequality2}
I_{M, \fq_2}(j, i)&\ll \int_{1/2-iT}^{1/2+iT}S^1_{j, i}(s)\frac{|A^0(M)|^{1/2}}{|s|^3}|ds|+\int_{1/2-iT}^{1/2+iT}S^2_{j, i}(s)\frac{|A^0(M)|}{|s|^3}|ds|\\
& +O\left(\frac{|A^0(M)|^{2}\sqrt{Y}(\log Y)^{d'}(\log N)^{d'}}{T^2 (\log N)^{B'}}\right).
\end{align}
Now using Cauchy-Schwarz inequality on $\chi$ sum and then again on $\fq$ sum we have
\begin{align*}
S^k_{j, i}(s)\leq \left(\sum_{2^jD_1<|\fq|\leq 2^{j+1}D_1}\frac{1}{\phi(\fq)}\sum_{\chi\neq \chi_{o}}|\tilde{g}_k(\chi, s)|^2\right)^{1/2}\left(\sum_{2^jD_1<|\fq|\leq 2^{j+1}D_1}\frac{1}{\phi(\fq)}\sum_{\chi\neq \chi_{o}}|\tilde{f}_i(\chi, s)|^2\right)^{1/2}.
\end{align*}
Therefore, using Lemma \ref{lem:main large sieve inequality} we have for $s=1/2+it(-T\leq t\leq T)$,
\begin{align*}
S^1_{j, i}(s)&\ll \left(\left(2^{j+1}D_1+\frac{2^i A_1}{2^j D_1}\right)\sum_{\xi; |\xi|\leq A_2}\frac{\tau(\xi)^C}{|\xi|}\right)^{1/2}\left(\left(2^{j+1}D_1+\frac{Y}{2^j D_1}\right)\sum_{w; |w|\leq A_2}\frac{\tau(w)^C}{|w|}\right)^{1/2}\\
& \ll \left(Y+2^i A_1\right)^{1/2}(\log N)^{d'}.
\end{align*}
Let us Choose
\[
Y:= (2^jD_1)^2 \quad \text{and} \quad T:= e^{2(\log |A^0(N)|)^4}.
\]
Using the above choice of $Y$ and $T$ we have
\[
S^1_{j, i}(s)\ll |A^0(N)|^{1/2}(\log N)^{-B'}.
\]
Let $2^RY< T\leq 2^{R+1}Y$ and for $0\leq r\leq R$
\[
\tilde{g}_2^{(r)}(\chi, s)=\sum_{2^rY< |w|\leq 2^{r+1}Y}\frac{\tilde{g}(w)\chi(w)}{|w|^s}.
\]
Therefore we have,
\begin{align*}
R\ll (\log N)^2 \quad \text{and} \quad \tilde{g}_2(\chi, s)=\sum_{r=0}^{R}\tilde{g}_2^{(r)}(\chi, s).
\end{align*}
Now using Lemma \ref{lem:main large sieve inequality} we have for $s=\sigma+it(-T\leq t\leq T)$,
\begin{align*}
S^2_{j, i}(s)&\ll \max_{0\leq r\leq R}\left(\left(2^{j+1}D_1+\frac{2^i A_1}{2^j D_1}\right)\sum_{\xi; |\xi|\geq  2^iA_1}\frac{\tau(\xi)^C}{|\xi|^2}\right)^{1/2}\\
&\times\left(\left(2^{j+1}D_1+\frac{2^rY}{2^j D_1}\right)\sum_{w; |w|\geq 2^r Y}\frac{\tau(w)^C}{|w|^2}\right)^{1/2}(\log N)^2\\
& \ll \left(\frac{2^j D_1}{A_1}+\frac{1}{2^jD_1}\right)^{1/2}\left(\frac{2^j D_1}{Y}+\frac{1}{2^jD_1}\right)^{1/2}(\log N)^{d'+2}\ll \log^{-B'}N.
\end{align*}
Using the above choice of $Y, T$ and Substituting above estimations into \eqref{Dyadic inequality2}, \eqref{Dyadic inequality} we have,
\begin{align*}
\Sigma_6{'}\ll |A^0(N)|\log^{-B'}N.
\end{align*}
 \end{proof}
 
 \section{Proof of Corollary \ref{corollary2}}
 \begin{proof}
 We need the following lemma.
 \begin{lemma}[Lemma 2, \cite{HIN2}]\label{Hinz lemma}
 If $|\fq|\ll \log ^{D}N$ with a positive constant $D$, then we have for a non-principal character $\chi \pmod \fq$
 \[
 \sum_{w\in A^0(N)}\chi(w)\ll |A^0(N)|\exp\left(-c(\log N)^{1/2}\right),
 \]
 for some $c=c(D, K)>0.$
 \end{lemma}
 Using Lemma \ref{Hinz lemma} we can say that the function $f(w)=\mathds{1}_{w}$ satisfies \eqref{SW condition number field} condition. Therefore under hypothesis that prime have level of distribution $1/2$, Corollary follows from Theorem \ref{thm:Motohashi number field version} and Corollary \ref{corollary1}.
 \end{proof}
 
\end{document}